\documentclass[10pt]{article}

\usepackage[a4paper, margin=3cm]{geometry}

\usepackage{amsmath}
\usepackage{amssymb}
\usepackage{amsthm}

\usepackage{xcolor}

\newcommand{\dd}{\, \mathrm{d}}
\newcommand{\ddiv}{\mathrm{div}\,}
\newcommand{\T}{\mathbb{T}^d}
\newcommand{\D}{\mathbb{D}}

\newtheorem{prop}{Proposition}
\newtheorem{lem}{Lemma}
\newtheorem{thm}{Theorem}
\newtheorem{rem}{Remark}

\numberwithin{equation}{section}
\numberwithin{thm}{section}
\numberwithin{lem}{section}
\numberwithin{prop}{section}

\title{Weak solutions for the Stokes system for compressible non-Newtonian fluids with unbounded divergence}
\author{Milan Pokorn\'y\footnote{Charles University, Faculty of Mathematics and Physics, Sokolovsk\'a 83, 186 75 Praha 8, Czech Republic, e-mail address: {\tt pokorny@karlin.mff.cuni.cz}.} \ \& Maja Szlenk\footnote{University of Warsaw, Faculty of Mathematics, Informatics and Mechanics, Banacha 2, 02-097 Warsaw, Poland, e-mail address: {\tt m.szlenk@uw.edu.pl}}}

\begin{document}

\maketitle

\begin{abstract}
    We investigate the existence of weak solutions to a certain system of partial differential equations, modelling the behaviour of a compressible non-Newtonian fluid for small Reynolds number. We construct the weak solutions despite the lack of the $L^\infty$ estimate on the divergence of the velocity field. The result was obtained by combining the regularity theory for singular operators with a certain logarithmic integral inequality for $BMO$ functions, which allowed us to adjust the method from \cite{feireisl_global_2015} to more relaxed conditions on the velocity.
\end{abstract}

\section{Introduction}

Our aim is to investigate the existence of weak solutions to equations, modelling a special case of compressible, non-Newtonian fluid. In the most general setting, the motion of  such a fluid without the presence of the external forces is described by the system of partial differential equations
\begin{equation}\label{ns}
\begin{aligned}
\varrho_t + \ddiv(\varrho u) = 0, \\
(\varrho u)_t +\ddiv(\varrho u\otimes u) -\ddiv\mathbb{S} =0,
\end{aligned} \end{equation}
where $\varrho$ is the density, $u$ is a velocity vector and $\mathbb{S}$ is the stress tensor; we assume that it is given by 
\[ \mathbb{S}(\mathbb{D}u,\varrho) = \mu\mathbb{D}u + \big(\lambda\ddiv u-p(\varrho)\big)\mathbb{I}, \]
where $\mu >0$ and $\lambda$ are the viscosity coefficients, $\mathbb{D}=\frac{1}{2}(\nabla+ \nabla^T)$ is the symmetric gradient and $p(\varrho)$ is the pressure. In the case of constant viscosity (i.e., the resulting system is called the compressible Navier--Stokes equations) $d\lambda + \mu\geq 0$, where $d$ is the space dimension.

We will focus on the case where the Reynolds number $\mathrm{Re}\sim \frac{\varrho |u|}{\mu}$ is small. As in this situation the advective forces are small compared to the viscous ones, we can approximate system (\ref{ns}) by the compressible Stokes-like system
\begin{equation}\label{stokes}
    \begin{aligned}
    \varrho_t + \ddiv(\varrho u) = 0, \\
    -\ddiv\mathbb{S} = 0.
    \end{aligned}
\end{equation}

Our aim is to obtain weak solutions to a special case of system (\ref{stokes}). We assume that the shear viscosity $\mu$ is in the form
\begin{equation}\label{mu} \mu=\mu_0(|\D u|)+2\mu_1, \quad \mu_1>0 \;\; \text{constant} \end{equation}
and the bulk viscosity $\lambda=\lambda(|\ddiv u|)$, where
\begin{equation}\label{viscosity}
    0\leq \mu_0(z),\lambda(z) \leq \frac{C}{z}, \quad z>0. 
\end{equation} 
Furthermore, the functions $\mathbb{R}^{d\times d}\ni B\mapsto\mu_0(|B|)B$ and $\mathbb{R}\ni s\mapsto \lambda(|s|)s$ are assumed to be monotone.

For the pressure we assume the barotropic case with $p(\varrho)=\varrho^\gamma$ for $\gamma\geq 1$. (Indeed, we could replace this precise form just by asymptotic growth conditions, similarly as in \cite{feireisl_global_2015}, i.e. $p(0) =0$, $p'(z) >0$ for $z >0$ and $\lim_{z\to \infty} \frac{p(z)}{z^\gamma} \in (0,\infty)$ for some $\gamma \geq 1$, but we skip such unnecessary complications). For simplicity we consider the space-periodic boundary conditions, hence $u\colon[0,T]\times\mathbb{T}^d\to\mathbb{R}^d$ and $\varrho\colon[0,T]\times\mathbb{T}^d\to\mathbb{R}$, where $\mathbb{T}^d$ is the $d$-dimensional torus. In conclusion, the analysed system of equations yield

\begin{equation}\label{main}
\begin{aligned} \varrho_t + \ddiv(\varrho u) &=0, \\
-\ddiv(\mu_0(|\D u|)\D u) - \mu_1\Delta u -\nabla((\mu_1+\lambda(\ddiv u))\ddiv u) + \nabla \varrho^\gamma &=0,
\end{aligned}
\end{equation}
with the initial condition
\begin{equation}
\varrho_{|_{t=0}} =\varrho_0\in L^\infty(\mathbb{T}^d) 
\end{equation}
and
\[ \int_{\T} u(t,x)\dd x = 0 \quad \forall_{t>0}. \]

Our system describes a fluid belonging to a class of power-law fluids. They are characterized by the behavior of the shear viscosity, which satisfies the relation
\begin{equation}\label{power-law} \mu\sim |\mathbb{D}u|^{r-2} \end{equation}
for some exponent $r\geq 1$. Typically, it is assumed that $\mu=\mu_0|\mathbb{D}u|^{r-2}$ or $\mu=\mu_0(a+|\mathbb{D}u|)^{r-2}$, $a>0$, to ensure that the viscosity is strictly positive and does not have singularities. For $r=2$, the fluid becomes Newtonian, whereas it is shear-thinning for $r<2$ and shear-thickening for $r>2$. The power-law fluids are used in many fields, for example glaciology \cite{man_significance_1987,kjartanson} and to analyze the dynamics in the Earth's Mantle \cite{yuen_strongly_1992} or blood flow \cite{cho_effects_1989,steffan_comparison_1990}. For more information we refer the reader, e.g., to \cite{Blechta_Malek_Raj}. Our situation corresponds specifically to a Herschel-Bulkley fluid, where the shear viscosity is in the form
\[\mu = \left\{ \begin{aligned}
 \mu_0, &\quad |\D u|< \delta, \\
 \frac{\tau_0}{|\D u|} + k|\D u|^{n-1}, &\quad |\D u|\geq \delta
\end{aligned}\right. \]
for some $n\geq 1$ and the parameters $\mu_0, \tau_0, k$ are chosen in such way that $\mu$ remains continuous. Fluids of this type were thoroughly analysed in the incompressible case, and have many industrial applications, see e.g.  \cite{H-B1,H-B2,H-B3}.

The mathematical theory concerning weak solutions to systems describing incompressible non-Newtonian fluids have been thoroughly developed in the past. There is a large number of papers dealing with several aspects of these problems. As it turns out, the existence and regularity of solutions to incompressible Navier--Stokes equations with the power-law relation (\ref{power-law}) for viscosity depends on the value of $r$. For $r>\frac{2d}{d+2}$ the existence of weak solutions for the problem
\begin{equation}\label{incompressible}
    \begin{aligned}
     \ddiv u &=0, \\
     u_t+\ddiv(u\otimes u) - \ddiv\mathbb{S} &= 0
    \end{aligned}
\end{equation}
with Dirichlet boundary conditions was shown for the first time in \cite{Die_Ruz_Wolf}; its uniqueness is known for $r\geq \frac{3d+2}{d+2}$, see \cite{Bu_Ka_etal}. As a matter of fact, the problem for $r < \frac{2d}{d+2} $ is ill-posed, see \cite{Bu_Mo_Sz}. However, existence of more general, dissipative solutions can be shown also in this case, see \cite{Abba_Fei}. 

Contrary to the incompressible case, the literature on the compressible non-Newtonian fluids is very limited. In \cite{mamontov1,mamontov2} Mamontov proved the existence of weak solutions to the system with linear pressure term and in the framework of Orlicz spaces with exponential growth, see also \cite{And_Vod} for further properties of these solutions. The results for more general form of the pressure were obtained in \cite{feireisl_global_2015}, where the authors considered the system (\ref{ns}) with $\mu$ of the form (\ref{power-law}) and a special form of $\lambda$, which provided the $L^\infty$ bound on $\ddiv u$. Using the classical Lions \& Feireisl method \cite{feireisl,lions}, they proved the existence of weak solutions for the same range of $r$ as the uniqueness and regularity theory is developed for incompressible fluids ($r\geq \frac{11}{5}$ in three dimensions). The additional bound on the divergence was crucial to obtain the strong convergence of the density in the final limit passage.

\subsection{Main result and structure of the paper}
Since the definition of the weak solution for our system \eqref{stokes} is straightforward, we may directly formulate our main result.

\begin{thm}\label{main_th}
Let $\varrho_0\in L^\infty(\mathbb{T}^d)$, $\gamma \geq 1$ and let \eqref{viscosity} hold. Then for any $T>0$ there exists a weak solution to the system (\ref{main}), satisfying
\[ \|\nabla u\|_{L^2((0,T)\times\mathbb{T}^d)} + \|\ddiv u \|_{L^\infty(0,T;L^p)}+ \|\varrho\|_{L^\infty(0,T;L^p)}  \leq C(p,T), \]
for any $1\leq p<\infty$, where $C$  approaches $\infty$ if $p$ or $T$ does so. 
\end{thm}

The proof of the theorem uses the technique from \cite{feireisl_global_2015}. However, due to the absence of the convective term we are able to obtain the result for relaxed assumptions on $\ddiv u$. In particular, we do not have from the very beginning that $\ddiv u$ is bounded. Instead of it, we obtain the $BMO$ regularity in space of the term
\[ (\mu_1+\lambda)\ddiv u-\varrho^\gamma. \]
This allows us to replicate the main step in the limit passage by using certain integral inequality from \cite{szlenk_weak_2022}.

\subsection{Preliminaries}

In the paper, we use standard notation for the Lebesgue and Sobolev space as well as the corresponding norms. For the notational simplicity, we omit the subscript while integrating over the torus, namely
\[ \int \dd x \ := \ \int_{\mathbb{T}^d} \dd x. \]
By $\{\cdot\}_Q$ we denote the mean integral over $Q$, while in the case of the whole torus we again omit the subscript. As the results do not depend on the values of $\mu_1$, for simplicity we set $\mu_1=1$.

A few other, slightly nonstandard spaces are defined below. 

\paragraph{Functions of bounded mean oscillation.} A function $f\in L^1(\Omega)$ belongs to space of bounded mean oscillation $BMO(\Omega)$ iff
 \[ \|f\|_{BMO} = \sup_{Q\subset\Omega}\frac{1}{|Q|}\int_Q |f-\{f\}_Q|\dd x < \infty, \]
 where the supremum is taken over all cubes in $\Omega$.

Note that $\|\cdot\|_{BMO}$ is not a norm, as $\|f\|_{BMO}=0$ for $f$ constant. However, we can endow the space $BMO(\Omega)$ with the norm
\[ \|\cdot\|_{L^1} + \|\cdot\|_{BMO} \]
and then it becomes the Banach space. The important result concerning $BMO$ functions in terms of our work is the following inequality: 
\begin{equation}\label{log_eq}
\begin{aligned} \left| \int_{\mathbb{T}^d}fg\dd x\right| \leq & C\|f\|_{BMO}\|g\|_{L^1} \\
& \times \big(|\ln\|g\|_{L^1}| + \ln(e+\|g\|_{L^q}) + (1+|\ln\|g\|_{L^1}|)\|g\|_{L^q}^{\frac{q-2}{2}}\big) \end{aligned} \end{equation}
for $g\in L^q$, $q>2$. The inequality was first obtained for $g\in L^\infty$ in \cite{mucha-rusin} and then adjusted in \cite{szlenk_weak_2022} for a wider range of functions. In \cite{de_rosa-inversi-stefani} the similar inequality was obtained for $f$ in the Orlicz space $L^{\exp}$.

The rest of the article is devoted to prove Theorem \ref{main_th}. First, in Section \ref{a_priori} we derive the a priori estimates, in particular the crucial $BMO$ estimate for the quantity $(2+\lambda(\ddiv u))\ddiv u-\varrho^\gamma$. Then, in Section \ref{appr_existence} we prove the existence of solutions to an approximated system with the regularized continuity and momentum equations. In Section \ref{compactness} we finish the proof by passing to the limit in the weak formulation of the approximate system and in consequence we obtain the weak solution of the original one. 

\section{A priori estimates}\label{a_priori}
\begin{lem}
Under the assumptions of Theorem \ref{main}, if the solution to (\ref{main}) is sufficiently smooth, it satisfies
\[ \|\nabla u\|_{L^2((0,T)\times\mathbb{T}^d)} + \|\varrho\|_{L^\infty(0,T;L^\gamma)} \leq C \]
and 
\[ \|(2+\lambda(\ddiv u))\ddiv u-\varrho^\gamma\|_{L^\infty(0,T;BMO)} \leq C \]
for $C$ depending only on $\|\varrho_0\|_\infty$. Furthermore,
\[ \|\ddiv u\|_{L^\infty(0,T;L^p)} + \|\varrho\|_{L^\infty(0,T;L^p)} \leq C(p,T) \]
for any $p<\infty$, where $C(p)\to \infty$ as $p\to\infty$ or $T\to \infty$.
\end{lem}
\begin{proof} Multiplying the second equation of (\ref{main}) by $u$ and integrating over the torus, we obtain (if $\gamma=1$, the last integral is replaced by $\int \varrho \ln\varrho \dd x$)
\[ \int \mu_0(|\D u|)|\D u|^2\dd x + \int |\nabla u|^2\dd x + \int (\ddiv u)^2\dd x +\int \lambda(\ddiv u)(\ddiv u)^2\dd x + \frac{1}{\gamma-1}\frac{\dd}{\dd t}\int \varrho^\gamma\dd x = 0. \]
Integrating the above equality from $0$ to $T$, we obtain the first desired estimate.

To obtain the $L^p$ estimate of the density, we use as test function in \eqref{stokes}$_2$ the function $\psi = -\Delta^{-1}\nabla \left(\varrho^{\theta}-\{\varrho^{\theta}\}\right)$ for some $\theta>1$. We have
\[ \ddiv\psi = \varrho^{\theta}-\{\varrho^{\theta}\}\]
and
\[ \|\nabla\psi\|_{L^r((0,T)\times\T)}\leq C(r)\|\varrho^{\gamma}\|_{L^r((0,T)\times\T)} \; \text{for any} \; 1<r<\infty. \]
Note that $C(r) \to \infty$ if $r\to 1^+$ or $r\to \infty$. Moreover,
\[ -\int \Delta u\cdot \psi \dd x = \int u\cdot \nabla \left(\varrho^{\theta}-\{\varrho^{\theta}\}\right) \dd x = -\int \varrho^\theta\ddiv u \dd x. \]

Then
\[
\begin{aligned}
\int \varrho^{\gamma +\theta} \dd x - 2 \int \varrho^{\theta}\ddiv u \dd x  &= -\int \mu_0(|\D u|)\D u:\nabla\psi\dd x  
- \int \varrho^\theta\lambda(\ddiv u)\ddiv u \dd x \\ & + \frac{1}{|\mathbb{T}^d|}\int \varrho^{\gamma}\dd x \int \varrho^\theta\dd x. 
\end{aligned}
\]
As  
\[\int \varrho^\theta\ddiv u \dd x\dd t = \frac{-1}{\theta-1}\frac{{\rm }d}{{\rm d}t}\int \varrho^\theta \dd x,\]
 we obtain
\[ 
 \frac{2}{\theta-1}\frac{{\rm }d}{{\rm d}t}\int \varrho^\theta \dd x + \int \varrho^{\gamma +\theta} \dd x \leq C \left(\int \varrho^\theta \dd x + \left(\int \varrho^{(1+\delta)\theta} \dd x\right)^{\frac 1{1+\delta}}\right);
 \]
 whence, for a suitably chosen $\delta$
 \[
 \|\varrho\|_{L^\infty(0,T;L^p} \leq C(p,T,\varrho_0).
 \]

The last estimate comes from the Calder\'on--Zygmund estimates. By taking the divergence of the momentum equation, we get
\[ \Delta((2+\lambda(\ddiv u))\ddiv u -\varrho^\gamma) = -\ddiv\ddiv\left(\mu_0(|\D u|)\D u\right). \]
Therefore in consequence
\[\begin{aligned} (2+\lambda(\ddiv u(t,x)))\ddiv u(t,x) -\varrho^\gamma(t,x) &= -\int \ddiv\ddiv\left(\mu_0(|\D u|)\D u\right)\bar K(x-y)\dd y \\
&= -\int \ddiv\left(\mu_0(|\D u|)\D u\right)\cdot \nabla\bar K(x-y)\dd y \\
&= -\int \mu_0(|\D u|)\D u:\nabla^2\bar K(x-y)\dd y, 
\end{aligned}\]
where $\Delta\bar K(x) =\delta_x$ in $\T$. Note that we can derive $\bar K$ by periodizing the fundamental solution to the Laplace equation on the whole space. If $K(x)=\frac{C_d}{|x|^{d-2}}$, then we simply put
\[ \bar K(x) = K(x) + \sum_{k\in \mathbb{Z}^d\setminus\{0\}} (K(x+k)-K(k)). \]

As $\nabla^2K$ is a singular kernel, so is $\nabla^2\bar K$. Therefore from the Calder\'on--Zygmund theorem we conclude that 
\[ \|(2+\lambda(\ddiv u))\ddiv u-\varrho^\gamma\|_{L^\infty(0,T;BMO)} \leq \left\|\mu_0(|\D u|)\D u\right\|_{L^\infty((0,T)\times\mathbb{T}^d)} \leq C \]
for some constant $C$ independent of $T$. Since $\varrho^\gamma$ is bounded in $L^\infty(0,T;L^p)$ and $\lambda(\ddiv u))\ddiv u$ in $L^\infty((0,T)\times \T)$, we finish the proof. 
\end{proof}

\section{Existence of approximate solutions}\label{appr_existence}
In this section we construct a unique solution to the system 
\begin{equation}\label{approx}
\begin{aligned}
 \varrho_t + \ddiv(\varrho u) +\delta\varrho^\beta &= \delta\Delta\varrho, \\
-\ddiv(\mu_0(\mathbb{D}u)\mathbb{D}u) - \Delta u - \nabla(1+\lambda(\ddiv u))\ddiv u + \nabla \varrho^\gamma &= -\varepsilon\Delta^{2m}u, \quad \int u(t,x)\dd x=0
\end{aligned}\end{equation}
for a sufficiently small $\delta,\varepsilon>0$,  sufficiently large $m\in \mathbb{N}$ and $\beta\geq \max\{\gamma+1,4\}$ being an even integer, 
 with the initial condition 
\[ \varrho_{|_{t=0}} = \varrho_{0,\delta}\in C^\infty(\T), \quad \varrho_{0,\delta}>0, \quad \varrho_{0,\delta} \to \varrho_0 \;\text{in any} \; L^p, \; p<\infty. \]

To prove the existence of solutions, we will employ the following version of the Schauder fixed point theorem:
\begin{thm}\label{Schauder}
Let $X$ be a Banach space and $\Phi\colon X\to X$ be continuous and compact. If the set 
\[ \{x\in X: \quad x=s\Phi(x) \quad \text{for some} \quad s\in [0,1]\} \]
is bounded, then $\Phi$ has a fixed point.
\end{thm}
Let us define the map $\Phi\colon C([0,T];L^{2\gamma})\to C([0,T];L^{2\gamma})$ in the following way:
\begin{enumerate}
\item For $\tilde{\varrho}\in C([0,T];L^{2\gamma})$, let $u$  be the unique solution to the equation
\begin{equation}\label{elliptic}
-\ddiv(\mu_0(|\D u|)\D u) - \Delta u -\nabla(1+\lambda(\ddiv u))\ddiv u + \nabla (\tilde\varrho)^\gamma = -\varepsilon\Delta^{2m}u, \quad \int u(t,x)\dd x=0.
\end{equation}

\item Then, let $\varrho$ be the solution to 
\begin{equation}\label{cont_epsilon} \varrho_t + \ddiv(\varrho u) +\delta\varrho^\beta = \delta\Delta\varrho, \quad \varrho_{|_{t=0}}=\varrho_{0,\delta}. \end{equation}

We set $\Phi(\tilde\varrho):= \varrho$. It is easy to see that the fixed point $\varrho$ and the corresponding $u$ solve our problem \eqref{approx}.
\end{enumerate}

First, let us show that the operator $\Phi$ is well-defined.

\begin{prop}
If $\tilde \varrho^\gamma \in L^\infty(0,T;L^2)$, then there exists a unique solution $u$ to equation (\ref{elliptic}), satisfying
\[ \|\nabla u\|_{L^\infty(0,T;L^2)} + \sqrt\varepsilon\|\Delta^m u\|_{L^\infty(0,T;L^2)} \leq C\|\tilde\varrho^\gamma\|_{L^\infty(0,T;L^2)}. \]
In particular, if $m$ is large enough, then
\[ \|u\|_{L^\infty(0,T;W^{1,\infty})} \leq \frac{C}{\sqrt\varepsilon}\|\tilde\varrho^\gamma\|_{L^\infty(0,T;L^2)}. \]
\end{prop}
\begin{proof}
By multiplying the equation by $u$ and integrating over the torus, we get
\[\begin{aligned} \int \mu_0(|\D u|)|\D u|^2 + |\nabla u|^2 + (1+\lambda(\ddiv u))(\ddiv u)^2 +&\varepsilon|\Delta^m u|^2 \;\dd x = \int \tilde\varrho^\gamma\ddiv u \;\dd x \\
&\leq \eta\int (\ddiv u)^2\dd x + \frac{C}{\eta}\|\tilde\varrho^\gamma\|_{L^\infty(0,T;L^2)}^2, \end{aligned}\]
hence picking $\eta$ small enough and taking supremum over time, we get the desired estimate. 

For existence, we consider the functional $I$ defined in $\overline{H^{2m}}(\T) = \{v\in H^{2m}(\T): \int v \dd x =0\}$ given by
\[ I[v] := \int \Big(F(\nabla v) + \frac{1}{2}|\nabla v|^2 + \Lambda(\ddiv v) + \frac{\varepsilon}{2}|\Delta^m v|^2 - \tilde \varrho^\gamma(t,\cdot) \ddiv v\Big) \dd x, \]
where $F$ satisfies 
\[ \frac{\partial}{\partial b_{i,j}}F(B) = \mu_0(|B|)b_{i,j}, \]
for $B=(b_{i,j})_{i,j}\in \mathbb{R}^{d\times d}$
and $\Lambda$ is such that $\Lambda'(s) = s+\lambda(s)s$. In particular, the assumptions on $\mu_0$ and $\lambda$ imply that $F$ and $\Lambda$ are convex and bounded from below.

From the definitions of $F$ and $\Lambda$ it follows that any minimizer of $I$ corresponds to a weak solution to (\ref{elliptic}). By the convexity of $F$ and $\Lambda$, the functional $I$ is convex. Moreover, for certain $C$ and $\eta<C$,
\[ I[v] \geq \varepsilon\|\Delta^m v\|_{L^2}^2 + C\|\nabla v\|_{L^2}^2 - \eta\|\nabla v\|_{L^2}^2 - \frac{C}{\eta}\|\tilde \varrho^\gamma(t,\cdot)\|_{L^2}^2 \geq C\|v\|_{H^{2m}}^2 - C\|\tilde\varrho^\gamma\|_{L^\infty(0,T;L^2)}^2, \]
and thus $I$ is coercive. Therefore $I$ has at a.e. time level a unique minimizer $v(t,\cdot)\in \overline{H^{2m}}(\T)$  and in consequence there exists a unique $u\in L^\infty(0,T;H^{2m})$ with zero mean value over the torus solving (\ref{elliptic}).
\end{proof}

Now we use the following classical result for the heat equation:
\begin{prop}\label{parabolic_th}
Let $h\in L^2(0,T;L^q)$ for $1<q<\infty$. Then the solution to
\[ \partial_t\varrho -\varepsilon\Delta\varrho = h, \quad \varrho_{|_{t=0}}=\varrho_0 \]
satisfies the estimate
\begin{equation}\label{parabolic_est} \varepsilon^{1/2}\|\varrho\|_{L^\infty(0,T;W^{1,q})} + \|\partial_t\varrho\|_{L^2(0,T;L^q)} + \varepsilon\|\varrho\|_{L^2(0,T;W^{2,q})} \leq C \left(\varepsilon^{1/2}\|\varrho_0\|_{W^{1,q}} + \|h\|_{L^2(0,T;L^q)}\right). \end{equation}
Moreover, if $h=\ddiv w$, $w\in L^2(0,T;L^q)$, then
\begin{equation}\label{parabolic_est2} \varepsilon^{1/2}\|\varrho\|_{L^\infty(0,T;L^q)} + \varepsilon\|\nabla\varrho\|_{L^2(0,T;L^q)} \leq C\left(\varepsilon^{1/2}\|\varrho_0\|_{L^q} + \|w\|_{L^2(0,T;L^q)}\right). \end{equation}
\end{prop}

From the previous proposition, we can also conclude
\begin{prop}\label{existence_rho}
If $u\in L^\infty(0,T;W^{1,\infty})$, then for equation (\ref{cont_epsilon}) there exists a unique nonnegative solution $\varrho\in L^\infty(0,T;W^{1,r})$ with $\partial_t\varrho \in L^2(0,T;W^{-1,r})$ for any $r<\infty$.
\end{prop}
\begin{proof}
We construct the solution $\varrho\in L^\infty(0,T;W^{1,2})$ by the Galerkin approximation. The nonnegativity of solutions is obtained by testing by negative part of $\varrho$ and is a conclusion of the fact that $\varrho^\beta\geq 0$. Then we improve the regularity by bootstrapping argument. We skip the details of these steps since they are based on classical arguments. Next, testing equation (\ref{cont_epsilon}) by $p\varrho^{p-1}$, we have 
\[ \frac{\dd}{\dd t}\int \varrho^p\dd x \leq (p-1)\|\ddiv u\|_{L^\infty}\int \varrho^p\dd x \]
and therefore 
\[ \|\varrho\|_{L^\infty(0,T;L^p)} \leq \|\varrho_{0,\varepsilon}\|_{L^p}e^{\frac{p-1}{p}\|u\|_{L^1(0,T;W^{1,\infty})}}. \]
Taking $p\to\infty$ we have $\varrho\in L^\infty((0,T)\times\T)$. In consequence $\varrho u\in L^2(0,T;L^r)$ for any $r<\infty$ and we can use (\ref{parabolic_est2}) to obtain
\[ \nabla\varrho \in L^2(0,T;L^r). \]
Employing the fact that $u\in L^\infty(0,T;W^{1,\infty})$, we have $\ddiv(\varrho u)\in L^2(0,T;L^r)$ and by (\ref{parabolic_est}) with $h= -\ddiv u - \delta \varrho^\beta$,
$\varrho\in L^\infty(0,T;W^{1,r})$ for any $r<\infty$, whereas the estimate for $\partial_t\varrho$ comes directly from the equation (\ref{cont_epsilon}). The uniqueness follows directly from the estimate
\[ \frac{\dd}{\dd t}\int (\varrho_1-\varrho_2)^2\dd x \leq C\int (\varrho_1-\varrho_2)\ddiv u \dd x, \]
where $\varrho_1$ and $\varrho_2$ are two possibly distinct solutions to problem \eqref{cont_epsilon}. 
\end{proof}

We will now show the properties of $\Phi$, which will allow us to apply directly the Schauder fixed point theorem (Theorem \ref{Schauder}).

\begin{prop}
The operator $\Phi$ is continuous and compact from $C([0,T];L^{2\gamma})$ to itself. Moreover, the set
\[ \{\varrho\in C([0,T];L^{2\gamma}): \varrho=s\Phi(\varrho) \quad \text{for some} \quad s\in[0,1] \} \]
is bounded.
\end{prop}
\begin{proof}
 Let $\tilde\varrho_1, \tilde\varrho_2\in C([0,T];L^{2\gamma})$ and $u_1,u_2$ be the corresponding solutions to (\ref{elliptic}). As before, denote $\Phi(\tilde\varrho_i)=\varrho_i$, $i=1,2$. 

\textbf{Compactness.} From the previous propositions we know that $\varrho\in L^\infty(0,T;W^{1,2\gamma})$ and $\partial_t\varrho \in L^2(0,T;W^{-1,2\gamma})$ and the bounds are uniform for bounded sets of $\tilde \varrho$ in the given spaces. Therefore, the compactness of $\Phi$ in $C([0,T];L^{2\gamma})$ follows straight from a variant of the Aubin--Lions lemma from \cite{Simon}. 

\textbf{Continuity.} We will estimate $u_1-u_2$ in terms of $\tilde\varrho_1-\tilde\varrho_2$.  We have
\begin{multline*} -\ddiv\Big(\mu_0(|\D u_1|)\D u_1-\mu_0(|\D u_2|)\D u_2\Big) -\nabla\Big((1+\lambda(\ddiv u_1))\ddiv u_1-(1+\lambda(\ddiv u_2))\ddiv u_2\Big) \\ -\Delta(u_1-u_2) +\varepsilon\Delta^{2m}(u_1-u_2) = -\nabla(\tilde\varrho_1^\gamma-\tilde\varrho_2^\gamma). \end{multline*}
Multiplying the above equality by $u_1-u_2$ and integrating over $\T$, we get
\[
\mathcal{A}(u_1,u_2) + \int \varepsilon|\Delta^m(u_1-u_2)|^2 + |\nabla(u_1-u_2)|^2 \;\dd x \\
= \int (\tilde\varrho_1^\gamma-\tilde\varrho_2^\gamma)(\ddiv u_1-\ddiv u_2)\dd x, 
\]
where 
\begin{multline*} \mathcal{A}(u_1,u_2) = \int \left(\mu_0(|\D u_1|)\D u_1-\mu_0(|\D u_2|)\D u_2\right):(\D u_1-\D u_2) \\
+ \Big((1+\lambda(\ddiv u_1))\ddiv u_1 - (1+\lambda(\ddiv u_2))\ddiv u_2\Big)(\ddiv u_1-\ddiv u_2) \dd x \geq 0 \end{multline*}
from the monotonicity of the functions $B\mapsto\frac{B}{a+|B|}$ and $s\mapsto\lambda(s)s$. In consequence, we have
\[ \begin{aligned}
  \|\nabla(u_1-u_2)\|_{L^2(\T)}^2 + &\varepsilon\|\Delta^m(u_1-u_2)\|_{L^2(\T)}^2 \leq \|\tilde\varrho_1^\gamma-\tilde\varrho_2^\gamma\|_{L^2(\T)}\|\ddiv(u_1-u_2)\|_{L^2(\T)} \\
 &\leq C\left(\|\tilde\varrho_1\|_{L^\infty(\T)}^{\gamma-1}+\|\tilde\varrho_2\|_{L^\infty(\T)}^{\gamma-1}\right)\|\tilde\varrho_1-\tilde\varrho_2\|_{L^2(\T)}\|\nabla(u_1-u_2)\|_{L^2(\T)} \\
&\leq C(\eta)\|\tilde\varrho_1-\tilde\varrho_2\|_{L^2(\T)}^2 + \eta\|\nabla(u_1- u_2)\|_{L^2(\T)}^2.
 \end{aligned}
 \]
Hence, choosing $\eta$ small enough,
\[ \begin{aligned} \|u_1-u_2\|_{L^\infty(0,T;W^{1,\infty})} \leq C\|\Delta^m(u_1-u_2)\|_{L^\infty(0,T;L^2)} & \leq C\|\tilde\varrho_1-\tilde\varrho_2\|_{L^\infty(0,T;L^2)} \\ & \leq C\|\tilde\varrho_1-\tilde\varrho_2\|_{L^\infty(0,T;L^{2\gamma})}. 
\end{aligned} \]
Moreover, $\varrho_1-\varrho_2$ satisfy
\begin{equation}\label{roznica} \partial_t(\varrho_1-\varrho_2) +\delta(\varrho_1^\beta-\varrho_2^\beta) -\delta\Delta(\varrho_1-\varrho_2) = -\ddiv(\varrho_1u_1-\varrho_2u_2) \end{equation}
with 
\[ (\varrho_1-\varrho_2)_{|_{t=0}}=0. \]

Let us now estimate $\|\varrho_1-\varrho_2\|_{L^\infty(0,T;L^p)}$. First, we write $\ddiv(\varrho_1u_1-\varrho_2u_2)$ as
\[ \ddiv(\varrho_1u_1-\varrho_2u_2) = u_1\nabla(\varrho_1-\varrho_2) + \nabla\varrho_2(u_1-u_2) + \ddiv u_1(\varrho_1-\varrho_2) + \varrho_2(\ddiv u_1-\ddiv u_2). \]
Then, multiplying (\ref{roznica}) by $p|\varrho_1-\varrho_2|^{p-2}(\varrho_1-\varrho_2)$ and integrating, we obtain
\[ \begin{aligned} \frac{\dd}{\dd t}\int |\varrho_1-\varrho_2|^p\dd x + & \delta p\int\left( |\varrho_1-\varrho_2|^{p-2}(\varrho_1^\beta-\varrho_2^\beta)(\varrho_1-\varrho_2) + (p-1)|\varrho_1-\varrho_2|^{p-2}|\nabla(\varrho_1-\varrho_2)|^2 \right)\dd x \\
=& -(p-1)\int |\varrho_1-\varrho_2|^p\ddiv u_1 \dd x \\
&- \int (\nabla\varrho_2(u_1-u_2) + \varrho_2(\ddiv u_1-\ddiv u_2))|\varrho_1-\varrho_2|^{p-2}(\varrho_1-\varrho_2) \dd x.  \end{aligned} \]
In consequence, as $(\varrho_1^\beta-\varrho_2^\beta)(\varrho_1-\varrho_2)\geq 0$,
\begin{multline*} \frac{\dd}{\dd t}\int |\varrho_1-\varrho_2|^p\dd x \leq (p-1)\|u_1\|_{W^{1,\infty}}\int |\varrho_1-\varrho_2|^p\dd x \\
+ \|\varrho_2\|_{W^{1,p}(\T)}\|u_1-u_2\|_{W^{1,\infty}(\T)}\|\varrho_1-\varrho_2\|_{L^p(\T)}^{p-1}. \end{multline*}
Therefore from Gronwall's lemma
\[ \|\varrho_1-\varrho_2\|_{L^\infty(0,T;L^p)} \leq C\|u_1-u_2\|_{L^\infty(0,T;W^{1,\infty})}, \]
where $C$ depends on $T$, $\|u_1\|_{L^1(0,T;W^{1,\infty})}$ and $\|\varrho_2\|_{L^2(0,T;W^{1,p})}$. In particular, 
\[ \|\varrho_1-\varrho_2\|_{L^\infty(0,T;L^{2\gamma})} \leq C\|u_1-u_2\|_{L^\infty(0,T;W^{1,\infty})} \leq C\|\tilde\varrho_1-\tilde\varrho_2\|_{L^\infty(0,T;L^{2\gamma})}. \]

\textbf{Estimates for the fixed points.} To complete the proof of the Proposition, we need to check if the points satisfying $\varrho=s\Phi(\varrho)$ are bounded in $L^\infty(0,T;L^{2\gamma})$ for any $s\in [0,1]$. Throughout the proof we will denote by $C$ various constants independent on $s$. We have for $s>0$ (if $s=0$, the proof is trivial)
\[ \frac 1s\partial_t\varrho  + \frac 1s\ddiv(\varrho u) + \frac{\delta}{s^\beta}\varrho^\beta = \frac 1s\delta\Delta\varrho \]
and 
\[ -\ddiv\left(\mu_0(|\D u|)\D u\right) -\Delta u-\nabla(1+\lambda(\ddiv u))\ddiv u +\varepsilon\Delta^{2m}u+ \nabla\varrho^\gamma =0. \] 
Multiplying the momentum equation by $u$ and integrating, we obtain analogously as for the a priori estimates
\begin{multline*} \|u\|_{L^2(0,T;W^{1,2})}^2 + \varepsilon\|\Delta^m u\|_{L^2((0,T)\times\T)}^2 + \|\varrho\|_{L^\infty(0,T;L^\gamma)}^\gamma \\
+ \frac{\delta}{s^{\beta-1}}\frac{\gamma}{\gamma-1}\int_0^T\int \varrho^{\beta+\gamma-1}\dd x\dd t +\delta\gamma \int_0^T \int |\nabla \varrho|^2 \varrho^{\gamma-2} \dd x \dd t \leq \int \varrho_{0,\varepsilon}^\gamma \dd x \leq C. \end{multline*}
Repeating the estimate from Proposition \ref{existence_rho}, we also get
\[ \|\varrho\|_{L^\infty(0,T;L^{2\gamma})} \leq \|\varrho_{0,\varepsilon}\|_{L^\infty(\T)}e^{\|u\|_{L^1(0,T;W^{1,\infty})}} \leq C. \]
\end{proof}
In consequence, the assumptions of Theorem \ref{Schauder} are satisfied and there exists at least one solution to (\ref{approx}) on $[0,T]\times\T$ for arbitrary $T>0$.

\section{Compactness}\label{compactness}

We will now prove that we can pass to the limit with $\delta,\varepsilon\to 0$ to obtain the solutions to system (\ref{main}). First, we will pass to the limit with $\varepsilon\to 0$ and then we improve the estimates on $\varrho$ uniform in $\delta$ and perform the second limit passage. Below, by $\overline{f}$ we will denote the weak limit of $f$.

\subsection{Limit passage with $\varepsilon\to 0$}

Let $(\varrho_{\delta,\varepsilon}, u_{\delta,\varepsilon})$ be a solution to (\ref{approx}). We have the following estimates uniform in $\varepsilon$ (here we use the lower bounds on $\beta)$:
\[ 
\begin{aligned}
\|u_{\delta,\varepsilon}\|_{L^2(0,T;W^{1,2})}^2 &+ \|\varrho_{\delta,\varepsilon}\|_{L^\infty(0,T;L^\gamma)}^\gamma + \delta\|\nabla\varrho_{\delta,\varepsilon}^{\gamma/2}\|_{L^2((0,T)\times\T)}^2 \\&+ \delta\|\varrho_{\delta,\varrho}\|_{L^{\gamma+\beta-1}((0,T)\times\T)}^{\gamma +\beta-1} + \delta \|\nabla \varrho_{\delta, \varepsilon}\|_{L^2((0,T)\times \T)}^2 \leq C. 
\end{aligned}
\]

In particular, at least up to a subsequence,
\[ u_{\delta,\varepsilon}\rightharpoonup u_\delta \quad \text{in} \quad L^2(0,T;W^{1,2}). \]
Moreover, as $\beta+\gamma-1\geq 2\gamma$, we know that
\[ \|\varrho_{\delta,\varepsilon}^\gamma\|_{L^2((0,T)\times\T)} \leq C(\delta) \]
and
\[ \|\varrho_{\delta,\varepsilon}u_{\delta,\varepsilon}\|_{L^p((0,T)\times\T)} \leq C(\delta) \]
for some suitable $p<2$. In consequence,
\[ \|\nabla \varrho_{\delta, \varepsilon}\|_{L^2((0,T)\times \T)}, \; \|\varrho_{\delta,\varepsilon}\|_{L^p((0,T)\times\T)}, \; \|\partial_t\varrho_{\delta,\varepsilon}\|_{L^p(0,T;W^{-1,p})} \leq C(\delta). \]
Therefore from the Aubin--Lions lemma $\varrho_{\delta,\varepsilon}\to \varrho_\delta$ in $L^p((0,T)\times\T)$, (at least up to a subsequence). Then we also have $\varrho_{\delta,\varepsilon}^\gamma\to \varrho_\delta^\gamma$ and $\varrho_{\delta,\varepsilon}^\beta \to\varrho_\delta^\beta$ in suitable $L^q$ spaces. In consequence, we are able to pass to the limit  in the continuity equation. For the momentum equation, first note that the regularizing term satisfies
\[ \varepsilon^{1/2}\|\Delta^m u_{\delta,\varepsilon}\|_{L^2((0,T)\times\T)}\leq C \]
and thus in the weak formulation
\[ \varepsilon\int_0^T\int \Delta^mu_{\delta,\varepsilon}\cdot \Delta^m\phi \dd x\dd t \to 0 \]
for $\phi\in C_0^\infty((0,T)\times\T)$.
Therefore in the weak formulation we obtain
\[ \int_0^T\int \overline{\mu_0(|\D u_\delta|)\D u_\delta}:\D\varphi + \nabla u_\delta:\nabla\varphi + \ddiv u_\delta\ddiv \varphi + \overline{\lambda(\ddiv u_\delta)\ddiv u_\delta}\ddiv \varphi - \varrho_\delta^\gamma\ddiv\varphi \dd x\dd t = 0. \]

Testing by $u_\delta$, we get
\begin{multline*} \int_0^T\int \overline{\mu_0(|\D u_\delta|)\D u_\delta}:\D u_\delta + |\nabla u_\delta|^2 + (\ddiv u_\delta)^2 + \overline{\lambda(\ddiv u_\delta)\ddiv u_\delta}\ddiv u_\delta -\varrho_\delta^\gamma\ddiv u_\delta \dd x\dd t  = 0. \end{multline*}
On the other hand,
\begin{multline*} \int_0^T\int \overline{\mu_0(|\D u_\delta|)|\D u_\delta|^2} + \overline{\lambda(\ddiv u_\delta)(\ddiv u_\delta)^2}\dd x\dd t +\limsup_{\varepsilon\to 0}\int_0^T\int |\nabla u_{\delta,\varepsilon}|^2 + (\ddiv u_{\delta,\varepsilon})^2\dd x\dd t \\
 - \int_0^T\int \varrho_\delta^\gamma\ddiv u \dd x\dd t \leq 0. \end{multline*}
Therefore using the monotonicity of $\mu_0(|\cdot|)\cdot$ and $\lambda(|\cdot|)\cdot$ and weak lower semicontinuity of the norm, we obtain the convergence $\nabla u_{\delta,\varepsilon}\to \nabla u_\delta$ in $L^2((0,T)\times\T)$, which allows us to pass to the limit in the remaining nonlinear terms.

\subsection{Limit passage with $\delta\to 0$}

Let $(\varrho_\delta,u_\delta)$ be the function obtained in the previous section, solving
\begin{equation}\label{approx2}
\begin{aligned}
\varrho_t + \ddiv(\varrho u) + \delta\varrho^\beta &= \delta\Delta\varrho, \\
-\ddiv((\mu_0(|\D u|)+1)\D u) - \nabla(\lambda(\ddiv u)\ddiv u) + \nabla\varrho^\gamma &= 0. \end{aligned} \end{equation}

Note that repeating the calculations from Section \ref{a_priori}, we get the estimate
\[ \|\ddiv u_\delta + \lambda(\ddiv u_\delta)\ddiv u_\delta -\varrho_\delta^\gamma\|_{L^\infty(0,T;BMO)} \leq C. \]
Moreover, using the uniform estimates on $\|u_\delta\|_{L^2(0,T;W^{1,2})}$ and $\|\varrho_\delta\|_{L^\infty(0,T;L^\gamma)}$, we will improve the integrability of $\varrho_\delta$ uniformly in $\delta$. 

Let $p>1$. Define the function $P_k(\varrho)$ as 
\[ P_k(\varrho) = \varrho\int_0^\varrho \frac{T_k(z)^p}{z^2} \dd z, \]
where $T_k\in C^\infty([0,\infty))$ is the truncation operator, namely $T_k(z)=z$ for $z<k$, $T_k(z)=k+1$ for $z>2k$, $T_k'(z) \geq 0$ as well as $T_k(z)\nearrow z$ as $k\to\infty$. It is easy to see that $P_k(\varrho)\nearrow \varrho^p$ as well. Using the renormalized continuity equation, we get
\begin{equation}\label{P_k} \frac{\dd}{\dd t}\int P_k(\varrho_\delta)\dd x + \delta\int\left(\varrho_\delta^\beta P_k'(\varrho_\delta) + \frac{pT_k(\varrho_\delta)^{p-1}T_k'(\varrho_\delta)}{\varrho_\delta}|\nabla\varrho_\delta|^2\right) \dd x = -\int T_k(\varrho_\delta)^p\ddiv u_\delta \dd x. \end{equation}
Now, let us test the momentum equation by the function 
\[ \psi = \Delta^{-1}\nabla\left(T_k(\varrho_\delta)^p - \{T_k(\varrho_\delta)^p\}\right). \]
We have
\[\begin{aligned} \int_0^T\int & \varrho_\delta^\gamma T_k(\varrho_\delta)^p \dd x \\ &\leq  C\big(\|\mu_0(|\D u_\delta|)\D u_\delta\|_{L^\infty((0,T)\times\T)}+\|\lambda(\ddiv u_\delta)\ddiv u_\delta\|_{L^\infty((0,T)\times\T)}\big)\|\nabla\psi\|_{L^{1+\frac{\gamma}{p}}((0,T)\times\T)} \\
&+ 2\int_0^T\int T_k(\varrho_\delta)^p\ddiv u_\delta\dd x + \|\varrho_\delta\|_{L^\infty(0,T;L^\gamma)}\int_0^T\int T_k(\varrho_\delta)^p \dd x. \end{aligned}\]
By Cauchy inequality,
\[\begin{aligned}
\int_0^T\int \varrho_\delta^\gamma T_k(\varrho_\delta)^p \dd x \leq & C\|T_k(\varrho_\delta)\|_{L^{p+\gamma}((0,T)\times\T)}^p + 2\int_0^T\int T_k(\varrho_\delta)^p\ddiv u_\delta \dd x \\
&\leq \eta\|T_k(\varrho_\delta)\|_{L^{p+\gamma}((0,T)\times\T)}^{p+\gamma} + C(\eta) + 2\int_0^T\int T_k(\varrho_\delta)^p\ddiv u_\delta \dd x.
\end{aligned}\]
As $ T_k(\varrho)^{p+\gamma}\leq  \varrho^\gamma T_k(\varrho)^p$, for sufficiently small $\eta$ we get
\[ \int_0^T\int \varrho_\delta^\gamma T_k(\varrho_\delta)^p\dd x\dd t - 2\int_0^T\int T_k(\varrho_\delta)^p\ddiv u_\delta \dd x \leq C. \]
Therefore using (\ref{P_k}), we get
\[ \begin{aligned}
&\int_0^T\int \varrho_\delta^\gamma T_k(\varrho_\delta)^p\dd x\dd t + \sup_{t\in (0,T]} \int P_k(\varrho_\delta(t,\cdot))\dd x \\ &+\delta\int\left(\varrho_\delta^\beta P_k'(\varrho_\delta) + \frac{pT_k(\varrho_\delta)^{p-1}T_k'(\varrho_\delta)}{\varrho_\delta}|\nabla\varrho_\delta|^2\right) \dd x \leq C (T,p). 
\end{aligned}\]
We pass to the limit with $k\to\infty$ using monotone convergence theorem and in consequence
\[ \|\varrho_\delta\|_{L^\infty(0,T;L^p)} \leq C(T,p) \quad \text{for any} \quad p<\infty. \]

Having that estimate, we are ready to pass to the limit with $\delta\to 0$.

From the estimates uniform in $\delta$, we know that in particular
\[ \begin{aligned} u_\delta \rightharpoonup u \quad &\text{in} \quad L^2(0,T;W^{1,2}), \\
\varrho_\delta \rightharpoonup^* \varrho \quad &\text{in} \quad L^\infty(0,T;L^\gamma), \\
 \varrho_\delta^\gamma \rightharpoonup^* \overline{\varrho^\gamma} \quad &\text{in} \quad L^\infty(0,T;L^p) \end{aligned}\]
and 
\[ \mu_0(\D u_\delta)\D u_\delta, \; \lambda(\ddiv u_\delta)\ddiv u_\delta \rightharpoonup^* \overline{\mu_0(|\D u|)\D u}, \; \overline{\lambda(\ddiv u)\ddiv u} \quad \text{in} \quad L^\infty((0,T)\times\T). \]
Moreover, $\|\ddiv u+\overline{\lambda(\ddiv u)\ddiv u}-\overline{\varrho^\gamma}\|_{L^\infty(0,T;BMO)}\leq C$.
Note that from the continuity equation it also follows that $\varrho_\delta\to\varrho$ in $C([0,T];W^{-1,r})$ for a suitable $r$, and in consequence $\overline{\varrho u}=\varrho u$. Having the above estimates and testing the continuity equation by $\varrho_\delta$, we also obtain
\[ \delta^{1/2}\|\nabla\varrho_\delta\|_{L^2((0,T)\times\T)} \leq C. \]
Then for $\phi\in C_0^\infty((0,T)\times\Omega)$, together with the estimate on $\|\varrho_\delta\|_{L^\infty(0,T;L^p)}$,
\[ \delta\int_0^T\int \left(\varrho_\delta^\beta\varphi +\nabla\varrho_\delta\cdot\nabla\phi\right)\dd x\dd t \to 0 \quad \text{with} \quad \delta\to 0. \]
In consequence, $(\varrho,u)$ satisfies the continuity equation (\ref{main}) in the renormalized sense. 

Next, we will pass to the limit in the momentum equation and apply an argument from \cite{feireisl_global_2015}.  Passing to the limit in the weak formulation, we get for any $\phi\in C_0^\infty((0,T)\times\T)$ and $t\leq T$
\begin{equation}\label{form_weak} \int_0^t\int \overline{\mu_0(|\D u|)\D u}:\D\phi + \nabla u\cdot\nabla\phi + \ddiv u\ \ddiv\phi + \overline{\lambda(\ddiv u)\ddiv u}\ \ddiv\phi \;\dd x\dd s = \int_0^t\int \overline{\varrho^\gamma}\ddiv\phi\dd x\dd s. \end{equation}
The regularity of $u$ allows us to put $\phi=u$ in (\ref{form_weak}) and then
\begin{equation}\label{do_zwartosci} \int_0^t\int \overline{\mu_0(|\D u|)\D u}:\D u + |\nabla u|^2 + (\ddiv u)^2 + \overline{\lambda(\ddiv u)\ddiv u}\ \ddiv u \;\dd x\dd s =\int_0^t\int \overline{\varrho^\gamma}\ddiv u \dd x\dd s. \end{equation}
On the other hand, the solutions to approximate equation (\ref{approx2}) satisfy
\begin{multline}\label{approx_energy}
\int_0^t\int \mu_0(|\D u_\delta|)|\D u_\delta|^2 \dd x\dd s +|\nabla u_\delta|^2 + (\ddiv u_\delta)^2 + \lambda(\ddiv u_\delta)(\ddiv u_\delta)^2 \;\dd x\dd s \\
+ \frac{1}{\gamma-1}\int \varrho_\delta^\gamma(t,\cdot)\dd x \leq \frac{1}{\gamma-1}\int \varrho_{0,\delta}^\gamma \dd x.
\end{multline}
Using the monotonicity of $\mu_0(|\D u|)\D u$ and $\lambda(\ddiv u)\ddiv u$, we know that
\begin{equation}\label{mono_limit} \overline{\mu_0(|\D u|)|\D u|^2} \geq \overline{\mu_0(|\D u|)\D u}:\D u \end{equation}
and
\begin{equation}\label{lambda_limit}
    \overline{\lambda(\ddiv u)(\ddiv u)^2} \geq \overline{\lambda(\ddiv u)\ddiv u}\ \ddiv u.
\end{equation}
Therefore, taking $\liminf_{\delta\to 0}$ in the energy inequality (\ref{approx_energy}), we obtain
\begin{multline*} \int_0^t\int \overline{\mu_0(|\D u|)\D u}:\D u + |\nabla u|^2 + (\ddiv u)^2 + \overline{\lambda(\ddiv u)\ddiv u}\ \ddiv u \;\dd x\dd s \\
+ \frac{1}{\gamma-1}\int \overline{\varrho^\gamma}(t,\cdot)\dd x \leq \frac{1}{\gamma-1}\int \varrho_0^\gamma \dd x. \end{multline*}
Comparing the last equation with (\ref{do_zwartosci}), we get
\[ \frac{1}{\gamma-1}\int \overline{\varrho^\gamma}(t,\cdot)\dd x -\frac{1}{\gamma-1}\int \varrho_0^\gamma \dd x \leq -\int_0^t\int \overline{\varrho^\gamma}\ddiv u \dd x\dd s. \]
We would like to estimate $\displaystyle\int \overline{\varrho^\gamma}(t,\cdot)- \varrho^\gamma(t,\cdot)\dd x$. As we already know that $\varrho$ satisfies the continuity equation in the renormalized sense, we have
\begin{equation} \label{4.4a} \frac{1}{\gamma-1}\int \varrho^\gamma(t,\cdot)\dd x - \frac{1}{\gamma-1}\int \varrho_0^\gamma \dd x = -\int_0^t\int \varrho^\gamma\ddiv u \;\dd x\dd s.
\end{equation}
Therefore
\[ \frac{1}{\gamma-1}\int \big(\overline{\varrho^\gamma}(t,\cdot)-\varrho^\gamma(t,\cdot)\big)\dd x \leq -\int_0^t\int \big(\overline{\varrho^\gamma}-\varrho^\gamma\big)\ddiv u \dd x\dd s. \]
We now use the fact that $\ddiv u+\overline{\lambda(\ddiv u)\ddiv u}-\overline{\varrho^\gamma}\in L^\infty(0,T;BMO)$ and the logarithmic inequality (\ref{log_eq}).
As $\overline{\varrho^\gamma}\geq \varrho^\gamma$ and $\varrho^\gamma,\overline{\varrho^\gamma}\in L^\infty(0,T;L^p)$ for any $p<\infty$, we have 
\[\begin{aligned} -\int_0^t\int \big(\overline{\varrho^\gamma}-\varrho^\gamma\big)\ddiv u \;\dd x\dd s = & -\int_0^t\int \big(\overline{\varrho^\gamma}-\varrho^\gamma\big)(\ddiv u + \overline{\lambda(\ddiv u)\ddiv u}-\overline{\varrho^\gamma})\dd x\dd s \\
&-\int_0^t\int \big(\overline{\varrho^\gamma}-\varrho^\gamma\big)\overline{\varrho^\gamma}\dd x\dd s + \int_0^t\int \big(\overline{\varrho^\gamma}-\varrho^\gamma\big)\overline{\lambda(\ddiv u)\ddiv u}\dd x\dd t \\
\leq & -\int_0^t\int \big(\overline{\varrho^\gamma}-\varrho^\gamma\big)(\ddiv u + \overline{\lambda(\ddiv u)\ddiv u}-\overline{\varrho^\gamma})\dd x\dd s \\
&+ \int_0^t\int \big(\overline{\varrho^\gamma}-\varrho^\gamma\big)\overline{\lambda(\ddiv u)\ddiv u} \dd x\dd s \\
\leq & C\int_0^t\int \big(\overline{\varrho^\gamma}-\varrho^\gamma\big)\dd x \left(\left|\ln\left(\int \big(\overline{\varrho^\gamma}-\varrho^\gamma\big)\dd x\right)\right| + 1\right)\dd s \\
&+ \|\overline{\lambda(\ddiv u)\ddiv u}\|_{L^\infty((0,T)\times\T)}\int_0^t\int \Big(\overline{\varrho^\gamma}-\varrho^\gamma\Big)\dd x\dd s,
\end{aligned}\]
where $C$ depends on $\|\ddiv u+\overline{\lambda(\ddiv u)\ddiv u}-\overline{\varrho^\gamma}\|_{L^\infty(0,T;BMO)}$ and $\|\overline{\varrho^\gamma}-\varrho^\gamma\|_{L^\infty(0,T;L^q)}$ for some $q>2$.
Thus denoting $y(t)=\int\big(\overline{\varrho^\gamma}-\varrho^\gamma\big)\dd x$, we have the inequality
\[ y(t) \leq C\int_0^t y(s)(|\ln y(s)|+1)\dd s \quad \text{with} \quad y(0)=0. \]
As the ordinary differential equation $z'=Cz(|\ln z|+1)$ has a unique solution, we can apply the comparison criterion and in consequence $y\equiv 0$ on the interval $[0,t_*]$ for some $t_*>0$. 
Then we can apply the same analysis on the consecutive intervals of length $t_*$ and in the end $y\equiv 0$ on the whole interval $[0,T]$. From this, as $\overline{\varrho^\gamma}\geq \varrho^\gamma$, it follows that in fact $\overline{\varrho^\gamma}=\varrho^\gamma$. Now taking again the limit in (\ref{approx_energy}) and subtracting (\ref{do_zwartosci}), we get due to fact that $\varrho^\gamma = \overline{\varrho^\gamma}$ and \eqref{4.4a}
\[\begin{aligned} 0 &\geq \int_0^t\int \Big(\overline{|\nabla u|^2}-|\nabla u|^2\Big) \dd x\dd s \geq \lim_{\varepsilon\to 0}\int_0^t\int |\nabla u_\varepsilon-\nabla u|^2\dd x\dd t. \end{aligned}\]
Therefore $\nabla u_\varepsilon \to \nabla u$ in $L^2((0,T)\times\Omega)$ and (for possibly another subsequence) the sequence converges also a.e. In consequence, by virtue of the Lebesgue dominated convergence theorem, 
\[ \overline{\mu_0(|\D u|)\D u} = \mu_0(|\D u|)\D u \]
and
\[ \overline{\lambda(\ddiv u)\ddiv u} = \lambda(\ddiv u)\ddiv u, \]
and thus $(\varrho, u)$ satisfies the weak formulation of the system (\ref{main}).

\begin{rem}
The assumptions on $\mu_0$ and $\lambda$ and the used method allows the situation when the viscosities are singular at $0$, e.g. $\mu_0=\frac{1}{|\D u|}$. Note, however, that in this case, while passing to the limit in the weak formulation, the term
$$
\int  \frac{\D u}{|\D u|}:\nabla \varphi \dd x
$$
is well defined by the values of $\D u$ provided $|\D u|>0$. For $|\D u|=0$ it is just defined as the corresponding limit, which is not necessarily equal to zero if $|\D u|$ is so, cf. e.g. \cite{LasMucha} in a similar context.
\end{rem}

{\bf Acknowledgement:} The work of M. P. was partially supported by the Czech Science Foundation, project No.: 22-01591S, whereas the work of M. S. was supported by National Science Centre grant 2018/29/B/ST1/00339. The results were obtained during the internship of M. S. in Prague.

\bibliographystyle{abbrv}
\bibliography{bibliografia.bib}

\end{document}